\documentclass[12pt, reqno]{amsart}

\usepackage[a4paper, centering, total={170mm,240mm}]{geometry}

\usepackage{amssymb,latexsym}
\usepackage{blindtext}
\usepackage[backgroundcolor=olive, linecolor=olive, textsize=tiny]{todonotes}
\usepackage{esint,dsfont}

\usepackage{palatino}

\usepackage[colorlinks=true, linkcolor=blue, citecolor=purple, urlcolor=blue]{hyperref}

\def\R{{\mathbb R}}
\def\N{\mathbb{N}}

\def\Lip{\operatorname{Lip}}
\def\Ker{\operatorname{Ker}}

\newtheorem{prop}{\bf Proposition}[section]
\newtheorem{thm}[prop]{\bf Theorem}
\newtheorem{cor}[prop]{\bf Corollary}
\newtheorem{lem}[prop]{\bf Lemma}
\newtheorem{rmk}[prop]{\it Remark}

\begin{document}

\title[Quasi-trees, Lipschitz free spaces, and actions on $\ell^1$]{{\bf\Large Quasi-trees, Lipschitz free spaces, and actions on $\ell^1$}}

\author[I. Vergara]{Ignacio Vergara}
\address{Departamento de Matem\'atica y Ciencia de la Computaci\'on, Universidad de Santiago de Chile, Las Sophoras 173, Estaci\'on Central 9170020, Chile}

\email{ign.vergara.s@gmail.com}
\thanks{This work is supported by the FONDECYT project 3230024 and the ECOS--ANID project 23003 \emph{Small spaces under action}}

\makeatletter
\@namedef{subjclassname@2020}{%
  \textup{2020} Mathematics Subject Classification}
\makeatother

%


\begin{abstract}
We show that the Lipschitz free space of a countable simplicial quasi-tree is isomorphic to $\ell^1$. As a consequence, every finitely generated group with Property (QT) of Bestvina--Bromberg--Fujiwara has a proper uniformly Lipschitz affine action on $\ell^1$ with quasi-isometrically embedded orbits. We also show that $3$-manifold groups admit proper uniformly Lipschitz affine actions on $\ell^1$.
\end{abstract}


\begingroup
\def\uppercasenonmath#1{} 
\let\MakeUppercase\relax 
\maketitle
\endgroup

\section{{\bf Introduction}}

The study of group actions on $L^1$ spaces has received increasing attention in recent years, motivated mainly by its connections with Kazhdan's Property (T) and the Haagerup property in the isometric case; see \cite{ChDrHa} for details. It has since become apparent that the more general framework of uniformly Lipschitz actions is much less restrictive and allows for quite different behaviours.

It was shown in \cite{Ver} that a group acting properly on a finite product of quasi-trees has a proper uniformly Lipschitz affine action on a subspace of $L^1$. This class includes residually finite hyperbolic groups and mapping class groups; see \cite{BeBrFu}. This result was later extended to all hyperbolic groups in \cite{Ver2}. In \cite{DruMac}, Dru\c{t}u and Mackay proved that residually finite hyperbolic groups and mapping class groups admit proper uniformly Lipschitz affine actions on $\ell^1$. Moreover, they obtain actions with quasi-isometrically embedded orbits for the former class, and a slightly weaker condition for the latter; see \cite[Theorem 1.6]{DruMac} and \cite[Theorem 1.8]{DruMac} for more precise statements.

More recently, Gartland \cite{Gar} was able to exploit a very interesting connection that exists between these topics and the realm of Lipschitz free Banach spaces. More concretely, he showed that the Lipschitz free space of a hyperbolic group --viewed as a pointed metric space-- is isomorphic to $\ell^1$. As a consequence, he obtains proper uniformly Lipschitz actions on $\ell^1$ with quasi-isometrically embedded orbits, extending Dru\c{t}u and Mackay's result to all hyperbolic groups. These ideas had been previously explored in \cite{CaCuDo}.

In this paper, we follow Gartland's strategy for groups acting on (products of) quasi-trees. A metric space $X$ is said to be a quasi-tree if it is quasi-isometric to a simplicial tree. If, in addition, $X$ is itself a graph, we say that it is a \emph{simplicial quasi-tree}.

Let $\ell^1$ denote the Banach space of sequences with summable absolute value. Our first result is the following; see Section \ref{S_Lipfree} for details on Lipschitz free spaces.

\begin{thm}\label{Thm_LFSQT}
The Lipschitz free space of a countable simplicial quasi-tree is isomorphic to $\ell^1$.
\end{thm}

There are two main ingredients in the proof of Theorem \ref{Thm_LFSQT}. The first one is Kerr's construction of rough isometries from quasi-trees into $\R$-trees; see \cite[Proposition 1.2]{Ker}. The second one is Godard's description of the Lipschitz free spaces of subspaces of $\R$-trees \cite{God}.

We point out that, for locally finite quasi-trees, Theorem \ref{Thm_LFSQT} can also be obtained from Kerr's construction, together with \cite[Theorem A]{Gar} and \cite[Lemma 3.4]{Gar}. However, for applications, the local finiteness assumption is too strong; see \cite[\S 10]{But} for a more detailed discussion of this fact. Hence we need a new argument in order to deal with locally infinite quasi-trees.

Next, we turn to group actions. Unless otherwise specified, whenever $X$ is defined as the product of $N$ metric spaces $(X_1,d_1),\ldots,(X_N,d_N)$, it will be endowed with the $\ell^1$ sum of the metrics:
\begin{align}\label{l1_metric}
d(x,y)=\sum_{i=1}^N d_i(x_i,y_i),
\end{align}
where $x=(x_1,\ldots,x_N)$ and $y=(y_1,\ldots,y_N)$. In the particular case when every $(X_i,d_i)$ is a connected graph endowed with the edge-path distance, the formula \eqref{l1_metric} gives exactly the edge-path distance on $X$ for its natural graph structure.

Every isometric action of a group on a metric space $X$ gives rise to an isometric action on its Lipschitz free space $\mathcal{F}(X)$; see Lemma \ref{Lem_act_F(X)}. This fact, together with Theorem \ref{Thm_LFSQT}, allows us to obtain the following.

\begin{thm}\label{Thm_act_l1}
Let $G$ be a finitely generated group acting by isometries on a product of quasi-trees $X=X_1\times\cdots\times X_N$, and let $o\in X$. Then $G$ admits a uniformly Lipschitz affine action $\sigma$ on $\ell^1$ such that
\begin{align*}
\|\sigma(s)0\|_1\geq 	\frac{1}{C}d_X(s\cdot o,o) - C,\quad\forall s\in G,
\end{align*}
for some constant $C>0$.
\end{thm}

We say that a finitely generated group has Property (QT) if it admits an isometric action on a product of quasi-trees such that the orbit maps are quasi-isometric embeddings; see \cite{BeBrFu}. This class is very rich, it contains residually finite hyperbolic groups \cite{BeBrFu}, mapping class groups \cite{BeBrFu}, and a significant subclass of $3$-manifold groups \cite{HaNgYa}; see also \cite{NguYan} for more examples. As a consequence of Theorem \ref{Thm_act_l1}, we obtain the following result, which partially answers Question 1.11 in \cite{DruMac}.

\begin{cor}\label{Cor_QT}
Let $G$ be a finitely generated group with Property (QT). Then $G$ has a proper uniformly Lipschitz affine action on $\ell^1$ with quasi-isometrically embedded orbits.
\end{cor}

As discussed above, for residually finite hyperbolic groups, this was already proved in \cite[Theorem 1.6]{DruMac}, and in \cite[Theorem E]{Gar} without the residual finiteness hypothesis. For mapping class groups, a slightly weaker statement was obtained in \cite[Theorem 1.8]{DruMac}. We point out that all the actions constructed in \cite{DruMac} have Lipschitz constant $2+\varepsilon$ for $\varepsilon>0$ arbitrary. Our methods do not yield such estimates.

Next we apply Theorem \ref{Thm_LFSQT} to acylindrically hyperbolic groups. In this case, Bralasubramanya's characterisation of acylindrical hyperbolicity \cite{Bal} allows one to recover the following result from \cite{DruMac}. Again, we do not obtain estimates on the Lipschitz constants as in \cite[Theorem 1.2]{DruMac}.

\begin{thm}[Dru\c{t}u--Mackay]\label{Thm_acyl}
Every acylindrically hyperbolic group admits a uniformly Lipschitz affine action on $\ell^1$ with unbounded orbits.
\end{thm}

As mentioned above, Corollary \ref{Cor_QT} can be applied to $3$-manifold groups. Let $M$ be a connected, compact, orientable $3$-manifold, and let $\pi_1(M)$ denote its fundamental group; we refer the reader to \cite[Part 3]{Mar} for details on $3$-manifolds. It was shown in \cite{HaNgYa} that $\pi_1(M)$ has Property (QT) if and only if no summand in its sphere-disk decomposition supports either $\operatorname{Sol}$ or $\operatorname{Nil}$ geometry. In that case, we can apply Corollary \ref{Cor_QT} to $\pi_1(M)$ directly and obtain actions on $\ell^1$ with quasi-isometrically embedded orbits. We are not able to conclude this in the presence of $\operatorname{Sol}$ or $\operatorname{Nil}$ geometry, but we still get properness.

\begin{thm}\label{Thm_3-man}
Let $M$ be a connected, compact, orientable 3-manifold. Then $\pi_1(M)$ has a proper uniformly Lipschitz affine action on $\ell^1$.
\end{thm}

This paper is organised as follows. Section \ref{S_prelim} consists of various preliminaries regarding quasi-trees and the different kinds of group actions that we consider in this paper. Section \ref{S_Lipfree} is devoted to Lipschitz free Banach spaces and a direct sum decomposition that will be crucial for our proofs. In Section \ref{S_R_trees}, we recall Kerr's construction of rough isometries into $\R$-trees in the particular case of simplicial quasi-trees. In Section \ref{S_Lip_qt}, we focus on the Lipschitz free space of a simplicial quasi-tree and the proof of Theorem \ref{Thm_LFSQT}. Next, in Section \ref{S_act_l1}, we apply these results to group actions and prove Theorem \ref{Thm_act_l1}, Corollary \ref{Cor_QT}, and Theorem \ref{Thm_acyl}. Finally, in Section \ref{S_3-man}, we focus on $3$-manifold groups and the proof of Theorem \ref{Thm_3-man}.

\subsection*{Acknowledgements}
I am grateful to Chris Gartland, Alice Kerr, and John Mackay for their very valuable comments and suggestions. I also thank Hoang Thanh Nguyen for kindly answering my questions about $3$-manifolds. A big part of this research was carried out at the \emph{Institute for Advanced Study in Mathematics of HIT}; I thank its members for their hospitality.

\section{{\bf Preliminaries}}\label{S_prelim}

\subsection{Quasi-trees}
Let $(X,d_X)$, $(Y,d_Y)$ be metric spaces. We say that $g:X\to Y$ is a quasi-isometric embedding if there is a constant $A\geq 1$ such that, for all $x_1,x_2\in X$,
\begin{align}\label{qi_embed}
\frac{1}{A}d_X(x_1,x_2)-A\leq d_Y(g(x_1),g(x_2))\leq A d_X(x_1,x_2)+A.
\end{align}
If, in addition, there is a constant $C>0$ such that every ball of radius $C$ in $Y$ intersects $g(X)$, then we say that $g$ is a quasi-isometry. In this case, the spaces $(X,d_X)$, $(Y,d_Y)$ are said to be quasi-isometric; see \cite[\S 8.1]{DruKap} for more details.

Let $X$ be a connected graph. We will view it as a metric space in two different ways. First, we can consider its set of vertices (or $0$-skeleton) $X^{(0)}$, endowed with the edge-path distance, which corresponds to the minimum of the lengths of paths joining each pair of vertices. This distance takes values in the natural numbers $\N$. The second approach is to view the graph as a ``continuous'' object by considering its $1$-skeleton $X^{(1)}$. In this case, we extend the edge-path distance to the edges by identifying each edge isometrically to the interval $[0,1]$. This new metric takes values in $\R_+$, and $X^{(0)}$ is isometrically embedded into $X^{(1)}$. Unless otherwise stated, we will always assume that $X^{(0)}$ and $X^{(1)}$ are endowed with these distances. Observe that these two spaces are quasi-isometric.

Recall that a tree is a connected graph without cycles. We say that a metric space is a quasi-tree if it is quasi-isometric to a tree. If a graph is a quasi-tree for its edge-path distance, we will say that it is a \emph{simplicial quasi-tree}. The following result says that, when dealing with actions of finitely generated groups, we can always restrict ourselves to simplicial quasi-trees; see \cite[Proposition 3.1]{Man} and \cite[Remark 3.2]{Man}.

\begin{prop}[Manning]\label{Prop_Manning}
Let $(X,d)$ be a quasi-tree, $o\in X$, and let $G$ be a finitely generated group acting by isometries on $X$. Then there is $r>0$ such that the set $S=\{s\in G\ \mid\ d(s\cdot o, o)\leq r\}$ generates $G$, and the Cayley graph $Y=\operatorname{Cay}(G,S)$ is a simplicial quasi-tree. Moreover, the map
\begin{align*}
s\quad\longmapsto\quad s\cdot o
\end{align*}
defines a quasi-isometric embedding of $Y$ into $X$.
\end{prop}

\begin{rmk}
The quasi-tree $Y$ given by Proposition \ref{Prop_Manning} will not be locally finite in general. In other words, the generating set $S$ need not be finite.
\end{rmk}

\subsection{Group actions on metric spaces}
Let $(X,d)$ be a metric space, and let $G$ be a group acting on it. We say that the action $G\curvearrowright X$ is uniformly Lipschitz if there is a constant $C\geq 1$ such that
\begin{align}\label{unif_Lip_act}
d(s\cdot x, s\cdot y) \leq C d(x,y),\quad\forall s\in G,\ \forall x,y\in X.
\end{align}
If this condition is satisfied with $C=1$, we say that the action is isometric.

Let $G\curvearrowright X$ be a uniformly Lipschitz action, and let $o\in X$. We say that the action is (metrically) proper if, for every $r>0$, the set
\begin{align*}
\{s\in G \ \mid\ d(s\cdot o, o)\leq r\}
\end{align*}
is finite. This condition does not depend on the choice of the point $o$.

In the case that the metric space is a Banach space $E$, we will be interested in actions by continuous affine transformations. More precisely, if $\operatorname{GL}(E)$ denotes the group of bounded, invertible operators on $E$, we look at actions of the form
\begin{align*}
\sigma(s)v=\pi(s)v+b(s),\quad\forall s\in G,\ \forall v\in E,
\end{align*}
where $\pi:G\to\operatorname{GL}(E)$ is a representation, i.e. a group homomorphism, and $b:G\to E$ is a map satisfying
\begin{align*}
b(st)=\pi(s)b(t)+b(s),\quad\forall s,t\in G.
\end{align*}
We say that $b$ is a cocycle for $\pi$. In this case, \eqref{unif_Lip_act} becomes
\begin{align*}
\sup_{s\in G}\|\pi(s)\|\leq C,
\end{align*}
and we say that $\pi$ is a uniformly bounded representation of $G$ on $E$.

\subsection{Induction of affine actions on $\ell^1$}
In order to prove Theorem \ref{Thm_act_l1}, we will need to induce actions from a finite-index subgroup to the ambient group. This is a standard procedure that we detail now in the particular case of actions on $\ell^1$.

Let $G$ be a group, and let $H$ be a finite-index subgroup of $G$. Let $\omega:G/H\to G$ be a section, i.e. a right inverse for the quotient map. We can always assume (and we will) that $\omega(H)$ is the identity element of $G$. Define $\alpha:G\times G/H\to H$ by
\begin{align}\label{cocy_fi_sg}
\alpha(s,x)=\omega(sx)^{-1}s\omega(x),\quad\forall s\in G,\ \forall x\in G/H.
\end{align}
Observe that, defining $\Omega=\omega(G/H)$, $G$ decomposes as
\begin{align}\label{G=bigsqcup}
G=\bigsqcup_{t\in H}\Omega t.
\end{align}
Then $\alpha(s,x)$ is the unique element of $H$ satisfying
\begin{align*}
s\omega(x)\in\Omega\alpha(s,x).
\end{align*}

Let $G$ be a group and let $e$ denote the identity element of $G$. We say that $|\cdot|:G\to[0,\infty)$ is a length function if $|e|=0$ and
\begin{align*}
|st|\leq |s| + |t|,\quad\forall s,t\in G.
\end{align*}
If $G$ acts by isometries on a metric space $(X,d)$, and we fix $o\in X$, then
\begin{align*}
|s|=d(s\cdot o,o)
\end{align*}
is a length function on $G$. In particular, if $S\subset G$ is a symmetric generating set, then $G$ acts on its Cayley graph $\operatorname{Cay}(G,S)$, and we obtain the associated word-length function
\begin{align*}
|s|_S=\min\{n\in\N\ \mid\ s\in S^n\},\quad\forall s\in G.
\end{align*}

The following is a very particular case of \cite[Proposition 3.4]{DruMac}, which gives an induction of actions from a lattice to the ambient locally compact group. Since we only deal with finite-index subgroups here, we include the proof of our particular case, which is much simpler.

\begin{lem}\label{Lem_fi_induc}
Let $G$ be a group, $H$ a finite-index subgroup of $G$, and $C\geq 1$. Let $|\cdot|$ be a length function on $G$, and assume that $H$ admits a $C$-Lipschitz affine action $\sigma$ on $\ell^1$ such that
\begin{align*}
\|\sigma(t)0\|_1\geq |t| - B,\quad\forall t\in H,
\end{align*}
for some constant $B>0$. Then there is a constant $\tilde{B}>0$ such that $G$ has a $C$-Lipschitz affine action $\tilde{\sigma}$ on $\ell^1$ satisfying
\begin{align*}
\|\tilde{\sigma}(s)0\|_1 \geq |s|-\tilde{B},\quad\forall s\in G.
\end{align*}
\end{lem}
\begin{proof}
Recall that $\sigma$ is given by
\begin{align*}
\sigma(s)v=\pi(s)v+b(s),\quad\forall s\in G,\ \forall v\in \ell^1,
\end{align*}
where $\pi:H\to\operatorname{GL}(\ell^1)$ is a uniformly bounded representation with
\begin{align*}
\sup_{t\in H}\|\pi(t)\| \leq C,
\end{align*}
and $b:H\to\ell^1$ is a cocycle satisfying
\begin{align*}
\|b(t)\|_1\geq |t|-B,\quad\forall t\in H.
\end{align*}
Let $E=\ell^1(G/H; \ell^1)$, namely the space of functions $f:G/H\to\ell^1$, endowed with the norm
\begin{align*}
\|f\|_E=\sum_{x\in G/H}\|f(x)\|_1.
\end{align*}
This space is isometrically isomorphic to $\ell^1$. Let $\alpha:G\times G/H\to H$ be as in \eqref{cocy_fi_sg}, and define $\tilde{\pi}:G\to\operatorname{GL}(E)$ by
\begin{align*}
(\tilde{\pi}(s)f)(x)=\pi(\alpha(s,s^{-1}x))f(s^{-1}x),\quad\forall s\in G,\ \forall f\in E,\ \forall x\in G/H.
\end{align*}
Since $\alpha$ satisfies
\begin{align*}
\alpha(st,x)=\alpha(s,tx)\alpha(t,x),\quad\forall s,t\in G,\ \forall x\in G/H,
\end{align*}
the map $\tilde{\pi}$ defines a representation. Moreover, for every $f\in E$,
\begin{align*}
\|\tilde{\pi}(s)f\|_E &=\sum_{x\in G/H}\|\pi(\alpha(s,s^{-1}x))f(s^{-1}x)\|_1\\
&\leq C \sum_{x\in G/H}\|f(s^{-1}x)\|_1\\
&= C \|f\|_E,
\end{align*}
where the last identity holds because $s^{-1}$ permutes the elements of $G/H$. Now define $\tilde{b}:G\to E$ by
\begin{align*}
\tilde{b}(s)(x)=b(\alpha(s,s^{-1}x)),\quad\forall s\in G,\ \forall x\in G/H.
\end{align*}
For all $s,t\in G$, $x\in G/H$,
\begin{align*}
\tilde{b}(st)(x) &= b(\alpha(s,s^{-1}x)\alpha(t,t^{-1}s^{-1}x))\\
&= \pi(\alpha(s,s^{-1}x))b(\alpha(t,t^{-1}s^{-1}x)) + b(\alpha(s,s^{-1}x))\\
&= \left(\tilde{\pi}(s)\tilde{b}(t)\right)(x) + \tilde{b}(s)(x),
\end{align*}
so $\tilde{b}$ is indeed a cocycle for $\tilde{\pi}$. Moreover, for all $s\in G$,
\begin{align*}
\|\tilde{b}(s)\|_E &=\sum_{x\in G/H}\|b(\alpha(s,s^{-1}x))\|_1\\
&\geq \|b(\alpha(s,s^{-1}H))\|_1 \\
&\geq |\alpha(s,s^{-1}H)|-B.
\end{align*}
Observe that $\alpha(s,s^{-1}H)=s\sigma(s^{-1}H)$, and therefore
\begin{align*}
|\alpha(s,s^{-1}H)|\geq |s| - D,
\end{align*}
where
\begin{align*}
D=\max\left\{|s|\ \mid\ s\in\Omega^{-1}\right\}.
\end{align*}
Defining $\tilde{\sigma}(s)f=\tilde{\pi}(s)f+\tilde{b}(s)$ and $\tilde{B}=B+D$, we obtain the desired result.
\end{proof}

\section{{\bf Lipschitz free spaces}}\label{S_Lipfree}

We now turn to Lipschitz free spaces. We quickly review the basic definitions; for a more detailed treatment, we refer the reader to \cite{Gode}.

Let $(X,d,o)$ be a pointed metric space, meaning that we fix a point $o\in X$ which we view as the origin. The space $\Lip_o(X)$ is given by all Lipschitz functions $f:X\to\R$ that vanish at $o$, endowed with the norm
\begin{align*}
\Lip(f)=\inf\left\{L>0\ \mid\ \forall x,y\in X,\ |f(x)-f(y)|\leq L d(x,y)\right\}.
\end{align*}
Let $\Lip_o(X)^*$ denote the dual Banach space of $\Lip_o(X)$. There is an isometric embedding $\delta:X\to\Lip_o(X)^*$ given by
\begin{align}\label{def_delta_x}
\langle\delta_x,f\rangle=f(x),\quad\forall x\in X,\ \forall f\in\Lip_o(X).
\end{align}
The Lipschitz free space $\mathcal{F}(X)$ is defined as the closed linear span of $\delta(X)$ in $\Lip_o(X)^*$. This space satisfies
\begin{align*}
\mathcal{F}(X)^*=\Lip_o(X)
\end{align*}
for the duality pairing \eqref{def_delta_x}.

When the metric space is a graph, the Lipshitz norm admits a simpler characterisation.

\begin{lem}\label{Lem_Lip_graph}
Let $X$ be a connected graph, and let $(X^{(0)},d)$ be its $0$-skeleton, endowed with the edge-path distance. Let $o\in X^{(0)}$, and $f\in\Lip_o(X^{(0)})$. Then
\begin{align*}
\Lip(f)=\sup\left\{|f(x)-f(y)|\ \mid\ x,y\in X^{(0)},\ d(x,y)=1\right\}.
\end{align*}
\end{lem}
\begin{proof}
Let $M=\sup\left\{|f(x)-f(y)|\ \mid\ d(x,y)=1\right\}$. By definition, $\Lip(f)\geq M$. Now let $x,y\in X^{(0)}$ with $d(x,y)=n$, and take a geodesic path $u_0,u_1,\ldots,u_n$ between $x$ and $y$, meaning that $u_0=x$, $u_n=y$, and $d(u_{i-1},u_i)=1$ for all $i\in\{1,\ldots,n\}$. Then
\begin{align*}
|f(x)-f(y)|\leq\sum_{i=1}^n|f(u_{i-1})-f(u_i)|\leq M n = M d(x,y),
\end{align*}
which shows that $\Lip(f)\leq M$.
\end{proof}

Given two Banach spaces $E,F$, we will use the notation $E\approx F$ to say that there is a bounded, bijective, linear map $\Phi:E\to F$. In this case, we say that $E$ and $F$ are isomorphic. If $\Phi$ can be chosen to be an isometry, then we say that $E$ and $F$ are isometrically isomorphic, and we write $E\cong F$.

Let $E_0$ be a closed subspace of $E$. We say that $E_0$ is complemented in $E$ if there is a bounded linear map $P:E\to E$ such that $P^2=P$ and $\operatorname{Ran}(P)=E_0$. In this case, we say that $P$ is a projection onto $E_0$. This is equivalent to the fact that $E_0$ admits a closed complement. More precisely,
\begin{align*}
E= E_0\oplus\Ker(P),
\end{align*}
where $\oplus$ denotes the algebraic direct sum of vector spaces; see \cite[\S 4.1]{FHHMZ} for details.

Given $p\in(1,\infty)$, we will use the notation $E\oplus_p F$ to denote the space $E\oplus F$, endowed with the norm
\begin{align*}
\|(u,v)\|=\left(\|u\|_E^p+\|v\|_F^p\right)^{1/p},\quad \forall u\in E,\ \forall v\in F.
\end{align*}
Similarly, $E\oplus_\infty F$ denotes the same space, endowed with the norm
\begin{align*}
\|(u,v)\|=\max\left\{\|u\|_E,\|v\|_F\right\},\quad \forall u\in E,\ \forall v\in F.
\end{align*}
The following result is probably well known to experts. Since it is essential to our purposes, we include its proof for completeness.

\begin{lem}\label{Lem_phi_g}
Let $(X,d_X)$, $(Y,d_Y)$ be metric spaces, and let $g:X\to Y$ be a Lipschitz map. Define $\phi_g:\mathcal{F}(X)\to\mathcal{F}(Y)$ by
\begin{align}\label{phi_g}
\phi_g(\delta_x)=\delta_{g(x)},\quad\forall x\in X.
\end{align}
Then $\phi_g$ is a well-defined, bounded linear map of norm at most $\Lip(g)$. Moreover, if $g$ is surjective and has a Lipschitz right inverse $h:Y\to X$, then $\mathcal{F}(Y)$ is isomorphic to a complemented subspace of $\mathcal{F}(X)$. More precisely,
\begin{align}\label{F(Y)_complm}
\mathcal{F}(X)\approx\mathcal{F}(Y)\oplus_1\Ker(\phi_g).
\end{align}
\end{lem}
\begin{proof}
For every finitely supported $u:X\to\R$, and every $f\in\Lip_{g(o)}(Y)$,
\begin{align}\label{f,phi_g(u)}
\langle f,\phi_g(u)\rangle = \sum_{x\in X}u(x)\langle f,\delta_{g(x)}\rangle = \langle f\circ g,u\rangle.
\end{align}
Hence
\begin{align*}
|\langle f,\phi_g(u)\rangle| \leq \Lip(f)\Lip(g)\|u\|_{\mathcal{F}(X)}.
\end{align*}
This shows that $\phi_g$ extends to a bounded linear map of norm at most $\Lip(g)$. Assume now that $g$ is surjective with a Lipschitz right inverse $h:Y\to X$. By the argument above, there is a bounded linear map $\phi_h:\mathcal{F}(Y)\to\mathcal{F}(X)$ satisfying $\phi_g\phi_h=\mathrm{id}_{\mathcal{F}(Y)}$. In particular, $\phi_h$ is injective. Then $\phi_h\phi_g:\mathcal{F}(X)\to\mathcal{F}(X)$ is a projection onto a subspace of $\mathcal{F}(X)$ which is isomorphic to $\mathcal{F}(Y)$. Moreover,
\begin{align*}
\Ker(\phi_h\phi_g)=\Ker(\phi_g),
\end{align*}
which gives the decomposition \eqref{F(Y)_complm}.
\end{proof}

As a consequence of Lemma \ref{Lem_phi_g}, we obtain a similar conclusion for the dual map $\phi_g^*:\Lip_{g(o)}(Y)\to \Lip_o(X)$.

\begin{cor}\label{Cor_phi_g*}
Let $(X,d_X)$, $(Y,d_Y)$ be metric spaces, and let $g:X\to Y$ be a Lipschitz map. Fix $o\in X$, and define $\phi_g^*:\Lip_{g(o)}(Y)\to \Lip_o(X)$ by
\begin{align*}
\phi_g^*(f)=f\circ g,\quad\forall f\in\Lip_{g(o)}(Y).
\end{align*}
Then $\phi_g^*$ is the adjoint map of $\phi_g:\mathcal{F}(X)\to\mathcal{F}(Y)$, given by Lemma \ref{Lem_phi_g}. Moreover, if $g$ is surjective and has a Lipschitz right inverse $h:Y\to X$ with $h(g(o))=o$, then $\Lip_{g(o)}(Y)$ is isomorphic to a complemented subspace of $\Lip_o(X)$. More precisely,
\begin{align}\label{Lip(Y)_complm}
\Lip_o(X)\approx\Lip_{g(o)}(Y)\oplus_\infty\Ker(\phi_h^*).
\end{align}
\end{cor}
\begin{proof}
The fact that $\phi_g^*$ is the adjoint map of $\phi_g$ follows from \eqref{f,phi_g(u)}. Then $\phi_g^*\phi_h^*:\Lip_o(X)\to\Lip_o(X)$ is a projection onto a subspace of $\Lip_o(X)$, which is isomorphic to $\Lip_{g(o)}(Y)$. Moreover,
\begin{align*}
\Ker(\phi_g^*\phi_h^*)=\Ker(\phi_h^*).
\end{align*}
\end{proof}

\begin{rmk}\label{Rmk_Ker_phi*}
In the decompositions \eqref{F(Y)_complm} and \eqref{Lip(Y)_complm}, there is a natural identification
\begin{align*}
\Ker(\phi_g)^*\approx\Ker(\phi_h^*).
\end{align*}
Indeed, in general, for a projection $P$ on a Banach space $E$, we have
\begin{align*}
\Ker(P)^*\approx\left(E/\operatorname{Ran}(P)\right)^*\cong\operatorname{Ran}(P)^\perp=\Ker(P^*);
\end{align*}
see e.g. \cite[Proposition 4.4]{FHHMZ}. Then apply this to $P=\phi_h\phi_g$.
\end{rmk}

\section{{\bf Kerr's construction of $\R$-trees}}\label{S_R_trees}

Let $X$ be a metric space. An arc in $X$ is the image of a topological embedding $[a,b]\to X$, where $[a,b]$ is a closed interval in $\R$. If this embedding can be chosen to be an isometry, then the arc is called a geodesic path.

An $\R$-tree is a geodesic metric space such that, for each pair of points, there is a unique arc joining them, and this arc is a geodesic path; see \cite{Bes} for details. We say that a point $x$ in an $\R$-tree $X$ is a branching point of $X$ if the subset $X\setminus\{x\}$ has at least 3 path-connected components. Observe that the $1$-skeleton of a tree is an $\R$-tree, and its branching points are exactly the vertices of degree greater than $2$.

In \cite{Ker}, Kerr showed that quasi-trees are roughly isometric to $\R$-trees. We review her construction here in the particular case of simplicial quasi-trees. We will see that the branching points of the associated $\R$-tree are contained in the image of the vertices of the quasi-tree. This fact will be crucial for the proof of Theorem \ref{Thm_LFSQT}.

Let $X$ be a simplicial quasi-tree, and let $X^{(0)}$, $X^{(1)}$ be its $0$-skeleton and $1$-skeleton respectively. We fix a point $o\in X^{(0)}$, and define, for all $x_1,x_2\in X^{(1)}$,
\begin{align*}
R_o(x_1,x_2)=\sup\{r\geq 0\ \mid\ x_1,x_2 \text{ lie in the same path component of } X^{(1)}\setminus B(o,r)\}.
\end{align*}
Here $B(o,r)$ denotes the open ball of radius $r$, centred at $o$. Observe that $R_o(x_1,x_2)$ is a natural number whenever $x_1$ and $x_2$ belong to $X^{(0)}$. We define a pseudo-metric on $X^{(1)}$ by
\begin{align}\label{def_d'}
d'(x_1,x_2)=d(x_1,o)+d(x_2,o)-2R_o(x_1,x_2),\quad\forall x_1,x_2\in X^{(1)}.
\end{align}
This is not Kerr's original definition (see \cite[Definition 3.1]{Ker}), but \cite[Lemma 6.3]{Ker} gives the equivalence with the formula above for geodesic metric spaces. Since $R_o$ is symmetric, so is $d'$. The fact that it satisfies the triangle inequality was proved in \cite[Lemma 3.4]{Ker}. We define $(Y,d_Y)$ as the quotient space
\begin{align}\label{Y=quot}
(Y,d_Y)=(X^{(1)},d')/\sim,
\end{align}
where
\begin{align*}
x_1\sim x_2 \iff d'(x_1,x_2)=0.
\end{align*}
The following was essentially proved in \cite[Proposition 3.8]{Ker} and \cite[Proposition 4.2]{Ker}.

\begin{lem}\label{Lem_Kerr+branch}
Let $X$ be a simplicial quasi-tree, and let $d_X$ denote the edge-path distance on $X^{(1)}$. Let $o\in X^{(0)}$ and let $(Y,d_Y)$ be defined as in \eqref{Y=quot}. Then $(Y,d_Y)$ is an $\R$-tree. Moreover, if $g:X^{(1)}\to Y$ denotes the quotient map, then there is a constant $\Delta>0$ such that
\begin{align}\label{ineq_d_T}
d_X(x_1,x_2)-\Delta\leq d_Y(g(x_1),g(x_2))\leq d_X(x_1,x_2),\quad\forall x_1,x_2\in X^{(1)},
\end{align}
and $g(X^{(0)})$ contains the set of branching points of $Y$.
\end{lem}
\begin{proof}
By \cite[Proposition 3.8]{Ker}, \cite[Proposition 4.2]{Ker} and \cite[Lemma 6.3]{Ker}, $(Y,d_Y)$ is an $\R$-tree, and the quotient map $g:X\to Y$ satisfies
\begin{align*}
d_X(x_1,x_2)-\Delta\leq d_Y(g(x_1),g(x_2))\leq d_X(x_1,x_2),\quad\forall x_1,x_2\in X^{(1)},
\end{align*}
for some constant $\Delta>0$. It only remains to show that $g(X^{(0)})$ contains all the branching points of $Y$. First observe that, setting $x_2=o$ in \eqref{def_d'}, we get
\begin{align*}
d_Y(g(x),g(o))=d_X(x,o),\quad\forall x\in X^{(1)},
\end{align*}
which shows that $g$ maps geodesic paths starting from $o$ into geodesic paths starting from $g(o)$. Now let $w\in Y$ be a branching point, and assume by contradiction that $w$ lies in $Y\setminus g(X^{(0)})$. Let
\begin{align*}
r=d_Y(g(o),w),
\end{align*}
which, by the discussion above, is not a natural number. Let $\bar{r}$ be the smallest natural number such that $r<\bar{r}$. Now let $A,B,C$ be three different path components of $Y\setminus\{w\}$ such that $g(o)\in A$, and let $u,v\in X^{(1)}$ such that $g(u)\in B$ and $g(v)\in C$. Fix $\varepsilon>0$ such that
\begin{align*}
r+\varepsilon<\min\left\{\bar{r},d_Y(g(o),g(u)),d_Y(g(o),g(v))\right\}.
\end{align*}
Since the only geodesic path joining $g(o)$ to $g(u)$ must pass through $w$, there exist $\tilde{w}_1,u'$ in a geodesic path joining $o$ to $u$ such that $g(\tilde{w}_1)=w$, $g(u')\in B$, and
\begin{align*}
d_Y(g(o),g(u'))=r+\varepsilon;
\end{align*}
see Figure \ref{Fig_Y-w}. Similarly, there are $\tilde{w}_2,v'$ in a geodesic path joining $o$ to $v$ such that $g(\tilde{w}_2)=w$, $g(v')\in C$, and
\begin{align*}
d_Y(g(o),g(v'))=r+\varepsilon.
\end{align*}
\begin{figure}[h]
\includegraphics[scale=1]{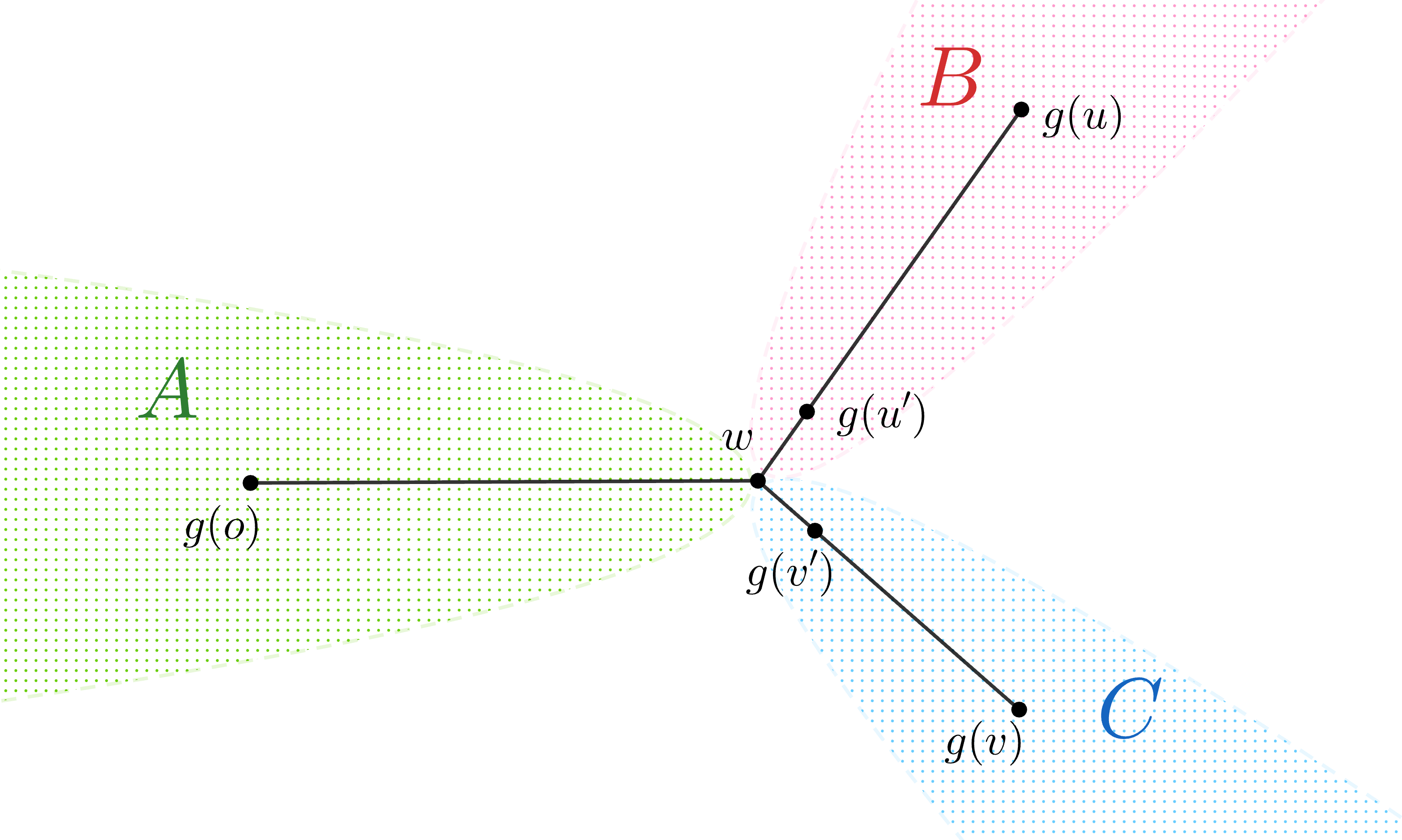}
\caption{Positions of $g(u'),g(v')$ in $Y\setminus\{w\}$.}
\label{Fig_Y-w}
\end{figure}
Observe that $\tilde{w}_1$ and $u'$ belong to the same edge $(a,b)$ in $X^{(1)}$, where
\begin{align*}
d_X(o,a)<r.
\end{align*}
Similarly, $\tilde{w}_2$ and $v'$ belong to the same edge $(c,d)$, and
\begin{align*}
d_X(o,c)<r.
\end{align*}
Hence $\tilde{w}_1\neq\tilde{w}_2$ because $u'\neq v'$. On the other hand, since $g(\tilde{w}_1)=g(\tilde{w}_2)$, we have
\begin{align*}
d'(\tilde{w}_1,\tilde{w}_2)=0,
\end{align*}
which means that there is a path $\gamma$ in $X^{(1)}\setminus B(o,r)$ joining $\tilde{w}_1$ to $\tilde{w}_2$. Then $\gamma$ must necessarily pass through $u'$ and $b$ because $d_X(o,a)<r$. Since $d_X(o,b)$ cannot be smaller than $r$, we have
\begin{align*}
d_X(o,b)=\bar{r}.
\end{align*}
Similarly,
\begin{align*}
d_X(o,d)=\bar{r}.
\end{align*}
Hence all the points lying in the portion of $\gamma$ joining $b$ to $d$ are at distance at least $\bar{r}$ from $o$ since otherwise there would be a vertex in $\gamma$ at distance $\bar{r}-1$ from $o$, which is impossible. This shows that
\begin{align*}
R_o(u',v')=d_X(o,u')=d_X(o,v'),
\end{align*}
which in turns shows that
\begin{align*}
d'(u',v')=0.
\end{align*}
This is impossible because $g(u')\neq g(v')$. We conclude that $w$ must belong to $g(X^{(0)})$.
\end{proof}

\section{{\bf The Lipschitz free space of a quasi-tree}}\label{S_Lip_qt}

In this Section, we prove Theorem \ref{Thm_LFSQT}. First we need the following lemma.

\begin{lem}\label{Lem_norm_Ker_phih}
Let $X$ be a simplicial quasi-tree, and let $(Y,d_Y)$, $\Delta$ and $g:X^{(1)}\to Y$ be as in Lemma \ref{Lem_Kerr+branch}. Let $h:g(X^{(0)})\to X^{(0)}$ be a right inverse for $g|_{X^{(0)}}$. Then $h$ is a $(1+\Delta)$-Lipschitz map. Moreover, if $\phi_h$ is defined as in \eqref{phi_g}, then, for every $f\in\Ker(\phi_h^*)$,
\begin{align*}
\frac{1}{2}\Lip(f)\leq\|f\|_\infty\leq\Delta\Lip(f).
\end{align*}
In particular, $\Ker(\phi_h^*)$ is isomorphic to $\ell^\infty(X^{(0)}\setminus h(g(X^{(0)})))$.
\end{lem}
\begin{proof}
First observe that, by \eqref{ineq_d_T}, for every $y_1,y_2\in g(X^{(0)})$,
\begin{align*}
d_X(h(y_1),h(y_2))\leq d_Y(y_1,y_2) + \Delta.
\end{align*}
Since $d_Y$ takes integer values on $g(X^{(0)})$, this implies that
\begin{align*}
d_X(h(y_1),h(y_2))\leq (1+\Delta)d_Y(y_1,y_2),\quad\forall y_1,y_2\in g(X_0),
\end{align*}
so $h$ is indeed a Lipschitz map and $\Lip(h)\leq 1+\Delta$. By Lemma \ref{Lem_phi_g}, $\phi_h$ is well defined and $\|\phi_h\|\leq 1+\Delta$. Now let $f\in\Ker(\phi_h^*)$. This means that $f$ vanishes in $h(g(X^{(0)}))$. Hence, for every $x\in X_0\setminus h(g(X^{(0)}))$,
\begin{align*}
|f(x)| &= |f(x)-f(h(g(x)))|\\
&\leq \Lip(f) d_X(x,h(g(x)))\\
&\leq \Lip(f) \left(d_Y(g(x),g(x)) + \Delta\right)\\
&= \Lip(f)\Delta.
\end{align*}
Therefore $f$ can be identified with an element of $\ell^\infty(X\setminus h(g(X^{(0)})))$ such that
\begin{align*}
\|f\|_\infty\leq\Delta\Lip(f).
\end{align*}
On the other hand, for every $x_1,x_2\in X^{(0)}$ with $d(x_1,x_2)=1$,
\begin{align*}
|f(x_1)-f(x_2)|\leq |f(x_1)| + |f(x_2)| \leq 2 \|f\|_\infty.
\end{align*}
By Lemma \ref{Lem_Lip_graph}, we conclude that
\begin{align*}
\Lip(f)\leq 2\|f\|_\infty.
\end{align*}
\end{proof}

Now we are ready to prove Theorem \ref{Thm_LFSQT}.

\begin{proof}[Proof of Theorem \ref{Thm_LFSQT}]
Let $X$ be a countable simplicial quasi-tree. By Lemma \ref{Lem_Kerr+branch}, there exist an $\R$-tree $(Y,d_Y)$ and a $1$-Lipschitz map $g:X^{(0)}\to Y$ such that $g(X^{(0)})$ contains the set of branching points of $Y$. Moreover, by Lemma \ref{Lem_norm_Ker_phih}, $g$ admits a Lipschitz right-inverse $h:g(X^{(0)})\to X^{(0)}$ with $h(g(o))=o$. By Lemma \ref{Lem_phi_g},
\begin{align}\label{F(X)=F(g(X))+ker}
\mathcal{F}(X^{(0)})\approx \mathcal{F}(g(X^{(0)}))\oplus_1 \operatorname{Ker}(\phi_g).
\end{align}
On the other hand, since $Y$ is separable (because $X^{(0)}$ is countable), and $g(X^{(0)})$ is a countable, discrete subset of $Y$ containing all its branching points, by \cite[Corollary 3.4]{God}, $\mathcal{F}(g(X^{(0)}))$ is isometrically isomorphic to $\ell^1$. Finally, by Remark \ref{Rmk_Ker_phi*} and Lemma \ref{Lem_norm_Ker_phih},
\begin{align*}
\operatorname{Ker}(\phi_g)^*\approx\operatorname{Ker}(\phi_h^*)\approx \ell^\infty,
\end{align*}
which shows that $\operatorname{Ker}(\phi_g)\approx\ell^1$; see e.g. \cite{Gode2}. Hence \eqref{F(X)=F(g(X))+ker} becomes
\begin{align*}
\mathcal{F}(X^{(0)})\approx\ell^1\oplus_1\ell^1\cong\ell^1.
\end{align*}
\end{proof}

\section{{\bf Uniformly Lipschitz actions on $\ell^1$}}\label{S_act_l1}

Now we turn to the proofs of Theorem \ref{Thm_act_l1}, Corollary \ref{Cor_QT}, and Theorem \ref{Thm_acyl}. The following result was essentially proved in \cite[Theorem 7.2]{Gar}. Starting from an isometric action on a metric space $X$, it associates, in a canonical way, an isometric action on the Lipschitz free space $\mathcal{F}(X)$.

\begin{lem}\label{Lem_act_F(X)}
Let $(X,d,o)$ be a pointed metric space, and let $G$ be a group acting by isometries on $X$. Then $G$ has an isometric affine action $\sigma$ on $\mathcal{F}(X)$ such that
\begin{align*}
\|\sigma(s)0\|=d(s\cdot o,o),\quad\forall s\in G.
\end{align*}
\end{lem}
\begin{proof}
Let $V$ denote the dense subspace of $\mathcal{F}(X)$ spanned by $\{\delta_x\ \mid\ x\in X\}$. For every $s\in G$, we define $\pi(s):V\to V$ by
\begin{align*}
\pi(s)\delta_x=\delta_{s\cdot x}-\delta_{s\cdot o}, \quad\forall x\in X.
\end{align*}
Let $u:X\to\R$ be a finitely supported function, and let $f\in\Lip_o(X)$. For all $s\in G$,
\begin{align*}
\left\langle f,\pi(s)\left(\sum_x u(x)\delta_x\right)\right\rangle &= \sum_x u(x)\left(f(s\cdot x)-f(s\cdot o)\right)\\
&= \left\langle f_s,\left(\sum_x u(x)\delta_x\right)\right\rangle,
\end{align*}
where $f_s(x)=f(s\cdot x)-f(s\cdot o)$. Observing that
\begin{align*}
\frac{|f_s(x)-f_s(y)|}{d(x,y)}=\frac{|f(s\cdot x)-f(s\cdot y)|}{d(s\cdot x,s\cdot y)},
\end{align*}
we see that $\Lip(f_s)=\Lip(f)$. Therefore $\pi(s)$ extends to an isometry on $\mathcal{F}(X)$. Furthermore, for all $s,t\in G$ and $x\in X$,
\begin{align*}
\pi(st)\delta_x&=\delta_{st\cdot x}-\delta_{s\cdot o}+\delta_{s\cdot o}-\delta_{st\cdot o}\\
&=\pi(s)\left(\delta_{t\cdot x}-\delta_{t\cdot o}\right)\\
&=\pi(s)\pi(t)\delta_x,
\end{align*}
which shows that $\pi$ defines an isometric representation on $\mathcal{F}(X)$. Now define $b:G\to\mathcal{F}(X)$ by
\begin{align*}
b(s)=\delta_{s\cdot o},\quad\forall s\in G.
\end{align*}
Then, for all $s,t\in G$,
\begin{align*}
b(st)&=\delta_{st\cdot o}-\delta_{s\cdot o}+\delta_{s\cdot o}\\
&=\pi(s)b(t)+b(s).
\end{align*}
Hence $b$ is a cocycle for $\pi$. Finally, since $\delta_o=0$, for all $s\in G$,
\begin{align*}
\|b(s)\|=\|\delta_{s\cdot o}-\delta_o\|=d(s\cdot o,o).
\end{align*}
Defining 
\begin{align*}
\sigma(s)v=\pi(s)v+b(s),\quad\forall s\in G,\ \forall v\in\mathcal{F}(X),
\end{align*}
we obtain the desired conclusion.
\end{proof}

Now we are ready to prove our main results.

\begin{proof}[Proof of Theorem \ref{Thm_act_l1}]
Let $G$ be a finitely generated group acting properly by isometries on a product of quasi-trees $X=X_1\times\cdots\times X_N$, and let us fix a point $o=(o_1,\ldots,o_N)$  in $X$. By \cite[Proposition 10.5]{But}, $G$ has a finite-index subgroup $H$ such that the restriction of this action to $H$ preserves each factor. More precisely, for each $i$ in $\{1,\ldots,N\}$, there is an isometric action $H\curvearrowright X_i$ such that the action on the product is given by
\begin{align*}
s\cdot(x_1,\ldots,x_N) = (s\cdot x_1,\ldots,s\cdot x_N),\quad\forall s\in H,\ \forall (x_1,\ldots,x_N)\in X.
\end{align*}
Let $d_i$ denote the distance on $X_i$. By Proposition \ref{Prop_Manning}, there is a countable simplicial quasi-tree $(\tilde{X}_i,\tilde{d}_i)$ endowed with an isometric action of $H$, an $H$-equivariant map $g_i:\tilde{X}_i\to X_i$, and a constant $C_i>0$ such that
\begin{align*}
\tilde{d}_i(s\cdot\tilde{o}_i,\tilde{o}_i)\geq\frac{1}{C_i}d_i(s\cdot o_i,o_i)-C_i, \quad\forall s\in H,
\end{align*}
where $g_i(\tilde{o}_i)=o_i$. On the other hand, by Lemma \ref{Lem_act_F(X)}, $H$ has an isometric affine action $\sigma_i$ on $\mathcal{F}(\tilde{X}_i)$ such that
\begin{align*}
\|\sigma_i(s)0\|=\tilde{d}_i(s\cdot\tilde{o}_i,\tilde{o}_i),\quad \forall s\in H.
\end{align*}
We can therefore define an isometric affine action $\sigma$ on
\begin{align*}
E=\mathcal{F}(\tilde{X}_1)\oplus_1\cdots\oplus_1\mathcal{F}(\tilde{X}_N)
\end{align*}
by
\begin{align*}
\sigma(s)(v_1,\ldots,v_N)=(\sigma_1(s)v_1,\ldots,\sigma_N(s)v_N),\quad\forall s\in H,\ \forall (v_1,\ldots,v_N)\in E.
\end{align*}
Observe that
\begin{align*}
\|\sigma(s)0\|_E &= \sum_{i=1}^N \|\sigma_i(s)0\|_{\mathcal{F}(\tilde{X}_i)}\\
&= \sum_{i=1}^N \tilde{d}_i(s\cdot\tilde{o}_i,\tilde{o}_i)\\
&\geq \frac{1}{C}\left(\sum_{i=1}^N d_i(s\cdot o_i,o_i)\right) -NC\\
&= \frac{1}{C} d(s\cdot o,o) - NC,
\end{align*}
where $C=\max\{C_1,\ldots,C_N\}$. Now, by Theorem \ref{Thm_LFSQT}, $\mathcal{F}(\tilde{X}_i)$ is isomorphic to $\ell^1$. Hence
\begin{align*}
E\approx\ell^1\oplus_1\cdots\oplus_1\ell^1\cong\ell^1.
\end{align*}
Let $\Psi:E\to\ell^1$ be an isomorphism. Then $\sigma_\Psi=\Psi\sigma(\,\cdot\,)\Psi^{-1}$ defines a uniformly Lipschitz affine action of $H$ on $\ell^1$. Moreover, there is a constant $\tilde{C}>0$ such that
\begin{align*}
\|\sigma_\Psi(s)0\|_1\geq\frac{1}{\tilde{C}} d(s\cdot o,o) - \tilde{C},\quad\forall s\in H.
\end{align*}
Finally, by Lemma \ref{Lem_fi_induc}, the same holds for $G$.
\end{proof}

\begin{proof}[Proof of Corollary \ref{Cor_QT}]
Let $G$ be a group with Property (QT), and let $X$ be a product of quasi-trees endowed with an isometric action of $G$ such that the orbit maps are quasi-isometric embeddings. By Theorem \ref{Thm_act_l1}, there is a uniformly Lipschitz affine action $\sigma$ of $G$ on $\ell^1$, and a constant $C>0$ such that
\begin{align*}
\|\sigma(s)0\|_1\geq\frac{1}{C}d(s\cdot o,o) - C,\quad\forall s\in G,
\end{align*}
for some $o\in X$. Let $|\cdot|$ be a word-legth on $G$. Since the orbit maps are quasi-isometric embeddings, there is a constant $A>0$ such that
\begin{align*}
d(s\cdot o,o) \geq\frac{1}{A}|s| - A,\quad\forall s\in G,
\end{align*}
Therefore, for every $v\in\ell^1$, and every $s\in G$,
\begin{align*}
\|\sigma(s)v\|_1 &\geq \|\sigma(s)0\|_1 - \|\sigma(s)v-\sigma(s)0\|_1\\
&\geq \frac{1}{C}d(s\cdot o,o) - C - \frac{1}{B}\|v\|_1\\
&\geq \frac{1}{AC}|s|  - \left(\frac{A}{C} + C + \frac{1}{B}\|v\|_1\right),
\end{align*}
for some $B>0$, which exists because $\sigma$ is uniformly Lipschitz. This gives us one inequality in the definition of quasi-isometric embedding; see \eqref{qi_embed}. The other one is always satisfied; see e.g. \cite[Proposition 2.10]{GueKam}. We conclude that the orbit maps of $\sigma$ are quasi-isometric embeddings.
\end{proof}

\begin{proof}[Proof of Theorem \ref{Thm_acyl}]
Let $G$ be an acylindrically hyperbolic group. By \cite[Theorem 1.7]{Bal}, $G$ has an isometric action on a quasi-tree with unbounded orbits. By Theorem \ref{Thm_act_l1}, this implies that $G$ has a uniformly Lipschitz affine action on $\ell^1$ with unbounded orbits.
\end{proof}

\section{{\bf $3$-manifold groups}}\label{S_3-man}

Now we focus on $3$-manifold groups and the proof of Theorem \ref{Thm_3-man}. The ideas developed in this section were kindly communicated to the author by John Mackay.

Let $M$ be a connected, compact, orientable $3$-manifold. It was shown in \cite{HaNgYa} that $\pi_1(M)$ splits as a free product of amenable groups and a group with Property (QT). The proof of Theorem \ref{Thm_3-man} follows from the fact that all these groups act properly on $\ell^1$.

Let $\Gamma$ and $\Lambda$ be two groups with identity elements $e_\Gamma$ and $e_\Lambda$ respectively. The free product $\Gamma\ast\Lambda$ is the group of reduced words with letters in $\Gamma\setminus\{e_\Gamma\}$ and $\Lambda\setminus\{e_\Lambda\}$, endowed with the concatenation product (with reduction); see e.g. \cite[\S 2.3.2]{Loh} for the formal definition. Every nontrivial element of $\Gamma\ast\Lambda$ can be written in a unique way as
\begin{align*}
s_1\cdots s_n,
\end{align*}
where each $s_i$ belongs to either $\Gamma\setminus\{e_\Gamma\}$ or $\Lambda\setminus\{e_\Lambda\}$, and no two consecutive letters belong to the same factor. The identity element of $\Gamma\ast\Lambda$ is the empty word.

Every free product $G=\Gamma\ast\Lambda$ admits a natural action on its Bass--Serre tree. We briefly describe this construction here; for more details, we refer the reader to \cite{Ser} or \cite[\S 4]{Yan}. The set of vertices is given by all the cosets of the form $s\Gamma$ and $s\Lambda$ for $s\in G$. The set of edges is identified with $G$. More precisely, all edges are of the form $\{s\Gamma,s\Lambda\}$ for $s\in G$. The graph thus defined is a tree, and $G$ acts on it by left multiplication.

The following result has appeared in several different forms in the literature; see e.g. \cite{AntDre} and \cite{ChaDah}. Since none of them fits exactly our setting, we include its proof here for completeness. Recall that, given a set $X$, a Banach space $E$, and $p\in[1,\infty)$, the space $\ell^p(X; E)$ is defined as the space of functions $f:X\to E$ such that the sum
\begin{align*}
\sum_{x\in X}\|f(x)\|_E^p
\end{align*}
is finite. Observe that, if $E=\ell^p$ and $X$ is countable, then
\begin{align*}
\ell^p(X; E)\cong \ell^p,
\end{align*}
for the natural norm on $\ell^p(X; E)$.

\begin{lem}\label{Lem_free_prod}
Let $\Gamma$ and $\Lambda$ be two countable groups admitting proper isometric affine actions on Banach spaces $E$ and $F$ respectively. Then, for every $p\in[1,\infty)$, the free product $\Gamma\ast\Lambda$ has a proper isometric affine action on the Banach space $\ell^p(T;E\oplus_p F\oplus_p\ell^p)$, where $T$ denotes the Bass--Serre tree of $\Gamma\ast\Lambda$.
\end{lem}
\begin{proof}
Let $G=\Gamma\ast\Lambda$. By hypothesis, there are isometric representations $\pi_E:\Gamma\to\operatorname{GL}(E)$, $\pi_F:\Lambda\to\operatorname{GL}(F)$, and proper cocycles $b_E:\Gamma\to E$, $b_F:\Lambda\to F$ associated to these representations. Let us now fix $p\in[1,\infty)$, and let $\lambda:\Gamma\times\Lambda\to\operatorname{GL}(\ell^p(\Gamma\times\Lambda))$ be the left regular representation:
\begin{align*}
\lambda(s_1,t_1)\xi(s_2,t_2)=\xi(s_1^{-1}s_2,t_1^{-1}t_2),\quad\forall s_1,s_2\in\Gamma,\ \forall t_1,t_2\in\Lambda,\ \forall\xi\in\ell^p(\Gamma\times\Lambda).
\end{align*}
Then we have an isometric representation of $\Gamma\times\Lambda$ on $E\oplus_p F\oplus_p\ell^p(\Gamma\times\Lambda)$ given by
\begin{align*}
(s,t)\in\Gamma\times\Lambda\quad\longmapsto\quad \pi_E(s)\oplus\pi_F(t)\oplus\lambda(s,t)\in \operatorname{GL}(E\oplus_p F\oplus_p\ell^p(\Gamma\times\Lambda)).
\end{align*}
Moreover, by concatenation, this gives an isometric representation $\pi:G\to\operatorname{GL}(E\oplus_p F\oplus_p\ell^p(\Gamma\times\Lambda))$. Now define
\begin{align*}
\tilde{E}=\ell^p(T;E\oplus_p F\oplus_p\ell^p(\Gamma\times\Lambda)),
\end{align*}
where $T$ denotes the Bass--Serre tree of $G$. Again, we can define an isometric representation $\tilde{\pi}:G\to\operatorname{GL}(\tilde{E})$ by
\begin{align*}
\tilde{\pi}(y)f(x)=\pi(y)f(y^{-1}x),\quad\forall y\in G,\ \forall f\in\tilde{E},\ \forall x\in T.
\end{align*}
Finally, we define $\tilde{b}:G\to\tilde{E}$ inductively as follows. For $s\in\Gamma$ and $x\in T$, we set
\begin{align*}
\tilde{b}(s)(x)=\begin{cases}
(b_E(s),0,\delta_s-\delta_{e_\Gamma},0), & \text{if } x=\Gamma,\\
(0,0,0,0), & \text{otherwise.}
\end{cases}
\end{align*}
Observe that the map $(s,t)\mapsto(\delta_s-\delta_{e_\Gamma},0)$ is a cocycle for $\lambda$. Similarly, for $t\in\Lambda$ and $x\in T$, define
\begin{align*}
\tilde{b}(t)(x)=\begin{cases}
(0,b_F(t),0,\delta_t-\delta_{e_\Lambda}), & \text{if } x=\Lambda,\\
(0,0,0,0), & \text{otherwise.}
\end{cases}
\end{align*}
Now assume that $\tilde{b}(y)$ has already been defined for some $y\in G$ such that its first letter is in $\Lambda$. Then, for every $s\in\Gamma$, we set
\begin{align*}
\tilde{b}(sy)=\tilde{\pi}(s)\tilde{b}(y)+\tilde{b}(s).
\end{align*}
If the first letter of $y$ is in $\Gamma$, we define $\tilde{b}(ty)$ analogously for every $t\in\Lambda$. This procedure allows us to define $\tilde{b}$ on every element of $G$. Moreover, it satisfies the cocycle identity by construction. Finally, by induction, one sees that
\begin{align*}
\|\tilde{b}(s_1\cdots s_n)\|_{\tilde{E}}^p=\sum_{i=1}^n\|b(s_i)\|^p + 2n,
\end{align*}
for every reduced word $s_1\cdots s_n$ in $G$, where
\begin{align*}
b(s_i)=\begin{cases}
b_E(s_i), & s_i\in\Gamma,\\
b_F(s_i), & s_i\in\Lambda.
\end{cases}
\end{align*}
This shows that $\tilde{b}$ is proper because both $b_E$ and $b_F$ are proper.
\end{proof}

No we apply Lemma \ref{Lem_free_prod} to uniformly Lipschitz actions on $\ell^p$. 

\begin{cor}\label{Cor_free_prod}
Let $p\in[1,\infty)$, and let $\Gamma$ and $\Lambda$ be two countable groups.
\begin{itemize}
\item[a)] If $\Gamma$ and $\Lambda$ admit proper isometric affine actions on $\ell^p$, then so does $\Gamma\ast\Lambda$.
\item[b)] If $\Gamma$ and $\Lambda$ admit proper uniformly Lipschitz affine actions on $\ell^p$, then so does $\Gamma\ast\Lambda$.
\end{itemize}
\end{cor}
\begin{proof}
Let us prove (b). By hypothesis, there are uniformly bounded representations $\pi_\Gamma:\Gamma\to\operatorname{GL}(\ell^p)$, $\pi_\Lambda:\Lambda\to\operatorname{GL}(\ell^p)$, and proper cocycles $b_\Gamma:\Gamma\to \ell^p$, $b_\Lambda:\Lambda\to \ell^p$. Let $E$ denote the space $\ell^p$ endowed with the equivalent norm
\begin{align*}
\|v\|_E=\sup_{s\in\Gamma}\|\pi_\Gamma(s)v\|_p.
\end{align*}
Then $\pi_\Gamma$ defines an isometric representation on $E$, and $b_\Gamma$ is still a proper cocycle for this representation. We define $F$ similarly, with the norm
\begin{align*}
\|v\|_F=\sup_{t\in\Lambda}\|\pi_\Lambda(t)v\|_p.
\end{align*}
By Lemma \ref{Lem_free_prod}, $\Gamma\ast\Lambda$ has a proper isometric affine action on $\ell^p(T; E\oplus_p F\oplus_p \ell^p)$, where $T$ is the Bass--Serre tree of $\Gamma\ast\Lambda$. Now observe that
\begin{align*}
E\oplus_p F\oplus_p \ell^p \approx \ell^p\oplus_p \ell^p\oplus_p \ell^p \cong \ell^p.
\end{align*}
This, in turn, shows that
\begin{align*}
\ell^p(T; E\oplus_p F\oplus_p \ell^p)\approx \ell^p.
\end{align*}
We conclude that $\Gamma\ast\Lambda$ has a proper uniformly Lipschitz affine action on $\ell^p$. The proof of (a) follows analogously by noticing that all the isomorphisms are isometric in that case.
\end{proof}

With this, we can prove Theorem \ref{Thm_3-man}.

\begin{proof}[Proof of Theorem \ref{Thm_3-man}]
Let $G=\pi_1(M)$, where $M$ is a connected, compact, orientable 3-manifold. By (the proof of) \cite[Theorem 1.1]{HaNgYa}, $G$ decomposes as $G=\Gamma\ast\Lambda$, where $\Gamma$ is a group with Property (QT), and
\begin{align*}
\Lambda=N_1\ast\cdots\ast N_k\ast S_1\ast\cdots\ast S_l,
\end{align*}
where $N_1,\ldots,N_k$ are virtually nilpotent groups, and $S_1,\ldots,S_l$ are solvable groups. In particular, these groups are amenable. By \cite[Theorem 42]{Now}, they admit proper isometric affine actions on $\ell^1$. Hence, by Corollary \ref{Cor_free_prod}(a), so does $\Lambda$. On the other hand, by Corollary \ref{Cor_QT}, $\Gamma$ has a proper uniformly Lipschitz affine action on $\ell^1$. Therefore, by Corollary \ref{Cor_free_prod}(b), $G$ has a proper uniformly Lipschitz affine action on $\ell^1$.
\end{proof}

\bibliographystyle{plain} 

\bibliography{Bibliography}

\end{document}